\documentclass[11pt]{article}

\usepackage[margin=1in]{geometry}
\usepackage{amsmath,amssymb,amsthm}
\usepackage{array}
\usepackage{booktabs}
\usepackage{longtable}   
\usepackage{microtype}
\usepackage[numbers]{natbib}
\usepackage[colorlinks=true,linkcolor=blue,citecolor=blue,urlcolor=blue]{hyperref}
\usepackage[capitalize]{cleveref}
\usepackage[T1]{fontenc}
\usepackage{lmodern}

\newtheorem{theorem}{Theorem}[section]
\newtheorem{proposition}[theorem]{Proposition}
\newtheorem{lemma}[theorem]{Lemma}
\theoremstyle{definition}
\newtheorem{definition}[theorem]{Definition}
\theoremstyle{remark}
\newtheorem*{remark}{Remark}

\providecommand{\tightlist}{\setlength{\itemsep}{0pt}\setlength{\parskip}{0pt}}
\renewcommand{\_}{\textunderscore\allowbreak}
\title{New lower bounds for the degree/diameter problem\\
       via interaction with a browser-accessible LLM}
\author{Ryosuke Mizuno\thanks{A.I.System Research Co., Ltd.\quad
        \texttt{mizuno@aisr.dev}}}
\date{11 June 2026}

\begin{document}
\maketitle

\begin{abstract}
Let \(N(\Delta,D)\) denote the maximum number of vertices in a simple
undirected connected graph with maximum degree at most \(\Delta\) and
diameter at most \(D\). Determining \(N(\Delta,D)\) is known as the
degree/diameter problem. Although the problem has been studied for many
years, the exact value of \(N(\Delta,D)\) is unknown for most pairs
\((\Delta,D)\). In this paper we construct explicit graphs, using a
construction discovered through interaction with
ChatGPT via its standard web interface, showing that \[
  N(12,5)\ge 34{,}992,\qquad
  N(16,5)\ge 147{,}456 .
\] These improve the corresponding recorded lower bounds \(29{,}621\)
and \(132{,}496\). The search was conducted without an external orchestration layer around
ChatGPT: no custom agent framework, automated evaluator-driven search
loop, problem-specific search engine, or formal proof assistant was set
up in advance by the author. As far as the visible
transcript shows, the author did not prompt the model with the concrete
components or candidate families of the construction.
After presenting the mathematical result, we describe the discovery
process on the basis of the visible transcript. We focus on the
meta-level interventions made during the approximately six-day search,
and we identify the stage at which the abstraction underlying the
construction first appeared.
\end{abstract}


\section{Introduction}\label{introduction}

Recent work has rapidly expanded the use of large language models
(LLMs), neural language models, and systems built around them for
mathematical construction search, theorem proving, problem generation,
and formalization support. Representative peer-reviewed examples include
FunSearch \cite{RomeraParedes2024}, which obtains mathematical
constructions by combining a pretrained LLM with a systematic evaluator
in a program-search loop, and AlphaGeometry \cite{Trinh2024}, which
combines a neural language model with a symbolic deduction engine for
olympiad geometry. Other examples include AlphaProof \cite{Hubert2026}, a
reinforcement-learning approach to formal proof search using Lean as a
verifiable environment, and TongGeometry \cite{Zhang2026}, which
addresses the proposal and proof of olympiad geometry problems by guided
tree search and a fine-tuned LLM.

From late 2025 through May 2026, preprints, technical reports, and
public reports from research institutions have also described the use of
AI/LLMs in mathematical discovery \cite{Bubeck2025, Woodruff2026,
NagdaRaghavanThakurta2026, BhanNobiliLanger2026, JuEtAl2026,
PrzybockiMackeyHeuleSubercaseaux2026, TsoukalasEtAl2026, OpenAI2026,
AlonEtAl2026}. Many of these examples rely on dedicated evaluators,
evolutionary search, tree search, formalizers, symbolic reasoning
engines, theorem retrieval, SAT solvers, the Lean prover, or
problem-specific evaluation loops. The present paper considers a
different and more minimal setting:

\begin{quote}
How far can mathematical search and discovery proceed using only an LLM
available through a standard browser interface?
\end{quote}

This question is practically relevant not only for laboratories able to
build large dedicated systems, but also for individual researchers and
developers who use LLMs for mathematical exploration. Even when a
dedicated agentic system or orchestration is eventually built, it is
useful to understand what kind of control is possible over a readily
available LLM, and how far that baseline can reach. As a case study in
this setting, this paper documents a new lower-bound construction for
the degree/diameter problem. The construction was obtained through a
long dialogue with ChatGPT in its standard web interface, during which
the author made meta-level interventions.

The degree/diameter problem asks for the maximum number of vertices in a
simple undirected graph with bounded maximum degree and bounded
diameter. In the notation used here, the maximum degree is at most
\(\Delta\), so each vertex has at most \(\Delta\) neighbors, and the
diameter is at most \(D\), so any two vertices are connected by a path
of length at most \(D\).

Intuitively, when each vertex has at most \(\Delta\) incident edges, the
number of vertices reachable within distance \(D\) from a fixed vertex
is maximized in the tree-like case, where no branches meet up to and
including depth \(D\). This gives the Moore upper bound \[
  N(\Delta,D)
  \le
  1+\Delta\sum_{i=0}^{D-1}(\Delta-1)^i
  =\frac{\Delta(\Delta-1)^D-2}{\Delta-2}
  \qquad(\Delta\ge3).
\] For many pairs \((\Delta,D)\) there is no Moore graph attaining this
bound, and the construction of large explicit examples has long been
studied \cite{HoffmanSingleton1960, Damerell1973, Bannai1973,
MillerSiran2013}. Known constructions include Cayley graphs, voltage
graphs, compound graphs, and other computational and algebraic
constructions, which have played an important role in setting records
for fixed parameters \cite{BermondDelormeFarhi1984,
BermondDelormeQuisquater1986, LozSiran2008, Comellas2026}.

The mathematical contribution of this paper is a construction, obtained
by integrating existing tools for the degree/diameter problem, that
establishes the new lower bounds \[
  N(12,5)\ge 34{,}992,
  \qquad
  N(16,5)\ge 147{,}456 .
\]

The main text first states the theorem, compares it with the recorded lower
bounds, and outlines the construction principle. Precise definitions,
proofs, and verification procedures are in Appendix A. We then describe
how the construction principle emerged during the dialogue, and how its
appearance was temporally related to failures, changes of direction, and
meta-level interventions by the author.

This paper does not treat ChatGPT as an author or as an autonomous
prover. ChatGPT was used as an interactive discovery engine. The author
checked the mathematical claims, code, certificates, citations, and
final manuscript, and takes responsibility for the contents of the
paper.

The LLM-generated files included in the supplementary materials are, in
principle, preserved in the form obtained during the dialogue, in order
to record the provenance of the discovery process and of the generated
construction. The mathematical claims of the paper are supported by the
explicit construction, the finite list of certificate conditions, and
the verification procedures that check them, all of which are given in
the main text and in Appendix A.
A cleaned, runnable version of the construction package, sharing the same graph
and certificate data, is also included in the supplementary materials so that the verification can be
reproduced directly, as described in \S\ref{a.8-verification-and-reproducibility} of Appendix A.

\subsection{Contributions of this
paper}\label{contributions-of-this-paper}

This paper makes two contributions.

\begin{enumerate}
\def\labelenumi{\arabic{enumi}.}
\item
  \textbf{Improved mathematical lower bounds.} For the degree/diameter
  problem with diameter \(5\), we give explicit simple undirected graph
  constructions establishing the new lower bounds \(N(12,5)\ge34{,}992\)
  and \(N(16,5)\ge147{,}456\).
\item
  \textbf{A record and analysis of a discovery process involving a
  browser-accessible LLM.} We describe a long dialogue with ChatGPT in
  its standard web interface, including unsuccessful search intervals.
  The dialogue was conducted without an external orchestration layer around
  ChatGPT: no custom agent framework, automated evaluator-driven search
  loop, problem-specific search engine, or formal proof assistant was set
  up in advance by the author. We also
  record, as process tracing rather than as a causal claim, that the
  author's meta-level interventions occurred near
  several junctures at which the search moved toward abstraction and
  finite certificate design. These interventions adjusted the
  constraints and acceptance criteria of the search.
\end{enumerate}

These two contributions can be read independently. The mathematical
result in Contribution 1 can be verified from the explicit construction
and the finite certificates in Appendix A alone. Its correctness does
not depend on the fact that the construction was discovered through
dialogue with a browser-accessible LLM. Contribution 2 is an
observational report of a single case concerning the process that led to
the discovery, and its outline can be understood without referring to
Appendix A.

\section{The degree/diameter problem and the main
result}\label{the-degreediameter-problem-and-the-main-result}

\subsection{Basic graph-theoretic terminology and the
degree/diameter
problem}\label{basic-graph-theoretic-terminology-and-the-degreediameter-problem}

Throughout, all graphs are finite. For integers \(\Delta,D\ge1\), define \[N(\Delta,D)\] to be the maximum number of vertices in a simple undirected connected
graph with maximum degree at most \(\Delta\) and diameter at most \(D\).
The degree/diameter problem is the problem of determining
\(N(\Delta,D)\). In the diameter-\(5\) case treated here, the goal is to
make the number of vertices as large as possible under the constraint
that any two vertices are joined by a path of length at most \(5\). 

We recall the underlying terminology, since this paper is intended to be
readable outside graph theory as well. The graphs considered here consist
of finitely many vertices and edges joining pairs of vertices.
 A graph whose edges are oriented is
called a \textbf{directed graph}, and a graph whose edges are not
oriented is called an \textbf{undirected graph}. An undirected graph
with no loops and no multiple edges is called a \textbf{simple
undirected graph}. A graph is \textbf{connected} if every vertex can be
reached from every other vertex by a path along edges.

The number of edges incident to a vertex \(v\) is the \textbf{degree} of
\(v\). Saying that the maximum degree of a graph is at most \(\Delta\)
means that every vertex has at most \(\Delta\) neighbors. A graph in
which all vertices have the same degree is called \textbf{regular}. The
\textbf{distance} between two vertices \(u,v\) is the minimum length of
a path from \(u\) to \(v\). The \textbf{diameter} of a graph is the
maximum distance over all pairs of vertices.
For example, constructing a \(12\)-regular graph of diameter \(5\) means
explicitly giving a large graph in which every vertex has degree \(12\)
and every pair of vertices is joined within \(5\) steps.

\subsection{Recorded lower bounds and the improvement in this
paper}\label{recorded-lower-bounds-and-the-improvement-in-this-paper}

Comellas's Mendeley Data V11 collects the largest known
\((\Delta,D)\)-graphs as of January 2026 \cite{Comellas2026,
ComellasData2026}. In the related web table, checked on 10 June 2026,
the entries relevant to this paper had the same lower-bound values. The
recorded lower bounds for diameter \(5\) relevant to this paper are as
follows.

\begin{table}[ht]
\centering
\caption{Recorded lower bounds for diameter~5 relevant to this paper.}
\label{tab:records}
\begin{tabular}{rrl}
\toprule
\((\Delta,D)\) & Recorded lower bound & Source \\
\midrule
\((12,5)\) & \(29{,}621\) & \cite{Comellas2025} \\
\midrule
\((16,5)\) & \(132{,}496\) & \cite{DelormeBipartite1985} \\
\bottomrule
\end{tabular}
\end{table}

This paper improves these to \(34{,}992\) and \(147{,}456\),
respectively.

\begin{theorem}\label{thm:main}
There exist explicit connected simple undirected graphs
\(G_3,G_4\) satisfying the following.

\begin{enumerate}
\def\labelenumi{\arabic{enumi}.}
\tightlist
\item
  \(G_3\) is \(12\)-regular, \(|V(G_3)|=34{,}992\), and
  \(\operatorname{diam}(G_3)=5\).
\item
  \(G_4\) is \(16\)-regular, \(|V(G_4)|=147{,}456\), and
  \(\operatorname{diam}(G_4)=5\).
\end{enumerate}
\end{theorem}

\cref{thm:main} immediately implies \[
  N(12,5)\ge34{,}992,
  \qquad
  N(16,5)\ge147{,}456 .
\]

The improvements are summarized below.

\begin{table}[ht]
\centering
\caption{Lower-bound improvements obtained in this paper, compared with the previously recorded lower bounds, together with the Moore upper bound and the resulting Moore ratio.}
\label{tab:known}
\begin{tabular}{crrrrr}
\toprule
Parameter & Previous LB & This paper & Improvement & Moore UB & Moore ratio \\
\midrule
\((12,5)\) & \(29{,}621\) & \(34{,}992\) & \(+5{,}371\) (\(+18.13\%\)) & \(193{,}261\) & \(18.11\%\) \\
\midrule
\((16,5)\) & \(132{,}496\) & \(147{,}456\) & \(+14{,}960\) (\(+11.29\%\)) & \(867{,}857\) & \(16.99\%\) \\
\bottomrule
\end{tabular}
\end{table}

Here the Moore ratio is the lower bound obtained in this paper divided
by the corresponding Moore upper bound.

\subsection{Outline of the construction: controller and
fiber}\label{outline-of-the-construction-controller-and-fiber}

The precise construction and proofs are given in Appendix A. This
section gives a guide to the construction principle.

A vertex of the final graph has the form \[
  (c,x)\in V(C)\times\mathbb F_q^5 ,
\] where \(\mathbb F_q\) is a finite field, \(c\) is a vertex of a small
auxiliary graph \(C\), and \(x\) is a coordinate in the copy of
\(\mathbb F_q^5\) lying above \(c\). This copy, \[
  \{c\}\times\mathbb F_q^5,
\] is called the \textbf{fiber} over \(c\).

We call \(C\) the \textbf{controller}. It determines the coarse routing
of the final graph. Edges are added to the vertex set
\(V(C)\times\mathbb F_q^5\) as follows.

\begin{enumerate}
\def\labelenumi{\arabic{enumi}.}
\item
  Choose a \(\rho\)-regular simple undirected graph \(C\). Let \(S\) be
  a finite set with \(|S|=\rho\), equipped with a fixed-point-free
  involution \(t\mapsto t^{-1}\), so that \[
    (t^{-1})^{-1}=t,\qquad t^{-1}\ne t .
  \] Such an \(S\) is called a symmetric alphabet.

  Regard each undirected edge of \(C\) as two directed edges, and assign
  an element of \(S\) to each directed edge. We require that, at each
  vertex \(u\in V(C)\), the \(\rho\) directed edges leaving \(u\)
  receive the elements of \(S\) exactly once, and that if a directed
  edge \(u\to v\) is assigned \(t\in S\), then the reverse directed edge
  \(v\to u\) is assigned \(t^{-1}\). Such a labeling is called an
  \textbf{inverse-consistent labeling}.
\item
  To each label \(t\in S\) assign \[
    A_t\in\operatorname{GL}_5(\mathbb F_q),
    \qquad
    b_t\in\mathbb F_q^5 ,
  \] with the inverse-symbol conditions \[
    A_{t^{-1}}=A_t^{-1},
    \qquad
    b_{t^{-1}}=A_t^{-1}b_t .
  \] This assignment is called a \textbf{route chart}. It specifies how
  fiber coordinates in \(\mathbb F_q^5\) move along the edges of \(C\).
\item
  Let \(t\in S\) be the label assigned to a directed edge \(u\to v\) of
  \(C\). For \((u,x),(v,y)\in V(C)\times\mathbb F_q^5\), join \((u,x)\)
  and \((v,y)\) by an edge if there exists \(\lambda\in\mathbb F_q\)
  such that \[
    y=A_t x+\lambda b_t .
  \] The graph obtained by lifting the edges of \(C\) to
  \(V(C)\times\mathbb F_q^5\) by this rule is called a
  \textbf{route-chart lift graph}.
\end{enumerate}

In a route-chart lift, the following two additional conditions guarantee
diameter at most \(5\).

\begin{enumerate}
\def\labelenumi{\arabic{enumi}.}
\item
  \textbf{Controller-side condition.} For any controller vertices
  \(c,d\in V(C)\), there is a nonbacktracking walk of length exactly
  \(5\) from \(c\) to \(d\). Here nonbacktracking means that the walk
  never immediately reverses the edge just traversed. A controller with
  this property is called an \textbf{exact-NB-5 controller}.
\item
  \textbf{Fiber-side condition.} A sequence of elements of \(S\) is
  called a \textbf{word}. A word \(w=t_1t_2t_3t_4t_5\) of length \(5\)
  is \textbf{reduced} if \[
    t_{i+1}\ne t_i^{-1}\qquad(1\le i<5).
  \] The fiber-side condition is that, for every reduced word
  \(w=t_1t_2t_3t_4t_5\) of length \(5\), the controllability matrix \[
    M_w=
    \bigl[
    A_{t_5}\cdots A_{t_2}b_{t_1}\mid
    A_{t_5}\cdots A_{t_3}b_{t_2}\mid
    A_{t_5}A_{t_4}b_{t_3}\mid
    A_{t_5}b_{t_4}\mid
    b_{t_5}
    \bigr]
    \in M_5(\mathbb F_q)
  \] is nonsingular. A route chart with
  this property is called a \textbf{universal route chart}.
\end{enumerate}

If the controller and fiber data satisfy these conditions, the following
proposition shows that the resulting graph has diameter at most \(5\).

\begin{proposition}[route-chart lift principle]\label{prop:lift}
Let \(q\) be a prime
power and \(s\ge1\) an integer. Let \(C\) be a
\(\rho\)-regular simple undirected graph and let \(S\) be a symmetric
alphabet with \(|S|=\rho\). Suppose that the directed edges of \(C\) are
given an inverse-consistent labeling. Suppose also that a universal
route chart on \(\mathbb F_q^s\) is given. If \(C\) is an exact-NB-\(s\) controller, then this data defines
a connected simple \(\rho q\)-regular route-chart lift graph \(G\) on \(V(C)\times\mathbb F_q^s\),
and \[
  |V(G)|=|V(C)|q^s,
  \qquad
  \operatorname{diam}(G)\le s .
\]
\end{proposition}

The proof is given in Appendix A. The key point is that the
controller-side condition and the fiber-side condition are independent.
The controller-side condition is a property of the controller alone,
independent of the fiber. The fiber-side condition depends only on the
route chart over the symbol set \(S\), independent of the controller.

Thus the problem of constructing a large graph with fixed degree and
diameter is reduced to two independent finite design problems.

\begin{table}[ht]
\centering
\caption{The two independent design problems in the route-chart lift construction.}
\label{tab:construction}
\begin{tabular}{>{\raggedright\arraybackslash}p{0.15\textwidth}>{\raggedright\arraybackslash}p{0.38\textwidth}>{\raggedright\arraybackslash}p{0.38\textwidth}}
\toprule
Component & Role & What to verify \\
\midrule
controller & a small graph determining the coarse routing of the final graph & every pair of controller vertices is joined by a length-\(5\) nonbacktracking walk \\
\midrule
fiber chart & linear-algebraic data moving fiber coordinates along controller routes & \(M_w\) is nonsingular for every reduced word of length \(5\) \\
\bottomrule
\end{tabular}
\end{table}

In the construction of this paper, the controller, the fiber chart, and
the final graph are built as follows.

First, from the \(6\)-vertex \(4\)-regular loopless multigraph \(B\) in
Appendix~\ref{a.5-the-144-vertex-controller}, together with a voltage assignment with values in \[
  H=C_4\times C_6,
\] we obtain a \(144\)-vertex simple \(4\)-regular connected graph \(C\)
by the standard voltage lift defined in Appendix~\ref{a.2-voltage-lifts-and-the-controller}. For this \(C\),
finite computation verifies that every ordered pair of vertices
\((c,d)\) is joined by a nonbacktracking walk of length exactly \(5\).
This graph \(C\) is the controller.

Next, we give the directed edges of \(C\) an inverse-consistent labeling
by \[
  S=\{a,a^{-1},b,b^{-1}\} .
\] We also give explicit route charts on \(\mathbb F_3^5\) and
\(\mathbb F_4^5\), and verify by finite computation that the
controllability matrix is nonsingular for every reduced word of length
\(5\).

Applying \cref{prop:lift} with \(q=3\) gives a \(12\)-regular graph on \[
  144\cdot3^5=34{,}992
\] vertices. Applying it with \(q=4\) gives a \(16\)-regular graph on \[
  144\cdot4^5=147{,}456
\] vertices. In both cases the diameter is at most \(5\).

Finally, each of these vertex counts exceeds the corresponding Moore
upper bound for diameter \(4\). Hence each of the two obtained graphs
has diameter exactly \(5\), and \cref{thm:main} follows.

The construction is closely related to the existing theory of graph
lifts and voltage graphs. The controller in \S\ref{a.5-the-144-vertex-controller} is an ordinary
group-valued voltage lift. The final route-chart lift can also be
described as a lift by a permutation voltage assignment after each
directed edge of the controller is expanded into \(q\) directed edges
indexed by \(\lambda\in\mathbb F_q\). The relation between the present
construction and existing methods is discussed in \S\ref{discussion}.

\section{Dialogue environment}\label{dialogue-environment}

The search was conducted by the author through ordinary dialogue in the
standard ChatGPT web interface. The author did not build or use an
external orchestration layer around ChatGPT: no custom agent framework,
automated evaluator-driven search loop, problem-specific search engine,
or formal proof assistant was set up in advance by the author. The setup
was as follows.

\begin{itemize}
\tightlist
\item
  User interface: ChatGPT in a web browser.
\item
  Model display name: \texttt{gpt-5-5-pro} in the transcript metadata.
\end{itemize}

In this paper, ``browser-accessible'' means that the external workflow
built and used by the author was limited to the standard ChatGPT web
interface. The internal implementation of the LLM, its inference
process, and any context-management mechanisms that may have been used
on the service side are not observable. The paper therefore makes no
claim that these mechanisms were simple.

All dialogue was conducted in Japanese. The user prompts quoted in
\S\ref{the-discovery-process}--\S\ref{transcript-based-observations-on-the-discovery-process} are the author's English translations. The original Japanese
messages are preserved in the supplementary visible transcript (Appendix
B).

Because the author could not rule out account-level memory,
personalization, or other cross-session context effects, this paper does
not claim that the main transcript released here was, by itself, the
sole source of information for the search. To the extent verified by the
author, the situation was as follows.

\begin{itemize}
\tightlist
\item
  Before Round 038 of the main session released here, other sessions
  contained exchanges on the degree/diameter problem for
  \((\Delta,D)=(8,2),(6,4)\) and on the degree-\(57\) Moore graph, as
  well as exchanges on other mathematical tasks.
\item
  However, no other session before Round 038 was confirmed to contain a
  general degree/diameter construction or the route-chart lift method
  ultimately used in this paper.
\item
  After Round 038, a separate session was started to verify the first
  lower-bound-improvement candidate obtained in Round 036.
\item
  After Round 042, in which the creation of handoff materials for other
  sessions was requested, two separate sessions in the same account
  began, in parallel, a search for generalizations of the result.
\item
  Neither separate session reached either the ideas that appeared in
  Rounds 049--050 of the released main session or the final lower bounds
  shown in \S\ref{the-degreediameter-problem-and-the-main-result}. The released main session was the first to reach the
  present result.
\end{itemize}

\section{The discovery process}\label{the-discovery-process}

The observations below concern Rounds 001--052 of the ChatGPT session,
conducted from 22 to 27 May 2026. These are the rounds relevant to the
present discovery and verification. They include the initial literature
survey, construction attempts, failure reports, changes of direction,
code generation, and verification results. Unrelated exchanges from
Round 053 onward are excluded from the analysis and from the
supplementary transcript.

The dialogue log was extracted from the conversation log downloaded
through the official ChatGPT export function. The supplementary script
\texttt{dump\_ddp\_visible.py} restricts the extraction to
browser-visible user prompts and assistant answers. In this paper, one
user prompt together with one or more responses to it is organized as a
round of the transcript. The full visible transcript of Rounds 001--052
is included in the supplementary materials. The structure of the
supplementary materials is summarized in Appendix B.

\subsection{The initial prompt}\label{the-initial-prompt}

The dialogue began in Round 001 with the author giving a high-level goal
of roughly the following form.

\begin{itemize}
\tightlist
\item
  Aim to set a new record for the degree/diameter problem within
  \(\Delta\le16\), \(D\le10\).
\item
  First, survey the relevant literature.
\item
  Then attempt a construction based on an idea not essentially contained
  in the literature.
\item
  Avoid mere combinations of existing methods, parameter increases,
  substitutions of the algebraic system, and non-novel local search
  around existing solutions.
\item
  During the idea-generation and construction phases, refrain from web
  search in principle.
\item
  The LLM may choose \(\Delta,D\) based on the suitability of its idea.
\end{itemize}

Except for specifying the target range of degree and diameter, the
initial prompt consisted of fairly general research instructions. It
asked the model to first conduct a literature survey, to attempt a proof
using an idea not essentially contained in the literature, and to
suppress web search during the ideation and construction stages. The
examples in the fourth item are specific to the degree/diameter problem,
but they were intended only to clarify what was meant by ``not
essentially contained'' in the literature.

Within the visible transcript analyzed in this paper, the author did not
directly specify the concrete components or candidate families of the
final construction. Instead, the author provided general constraints and
acceptance criteria, including differentiation from the literature,
suppression of web search, withdrawal from local repair, stopping and
reorganization, and requests to explain prospects of success in advance.
Candidate generation itself was left mainly to the LLM.

\subsection{Initial exploration}\label{initial-exploration}

In the initial exploration from Round 001 to Round 021, many ideas were
examined, including what the transcript labeled as non-regular
permutation lifts, thinning of finite geometries, matching contraction,
Hashimoto / de Bruijn type constructions, rank-controller type
constructions, random regular graphs, and affine Schreier cover type
constructions. Most of these did not ultimately lead to a certificate.

\subsection{The first essential transition: route as coordinate
chart}\label{the-first-essential-transition-route-as-coordinate-chart}

The first substantive progress occurred in Round 023, when the principle
of the controllable route-chart lift appeared. This principle is the
core of the final construction outlined in \S\ref{outline-of-the-construction-controller-and-fiber}. In Rounds 022 and 023,
the user requested ``an a priori mathematical reason to expect success''
and asked the LLM ``to devise a new principle.'' In Round 023, the LLM
presented the following view.

\begin{quote}
A base/controller walk is not merely a displacement. A controller route
of length \(k\) should be a coordinate chart that covers all of
\(\mathbb F_q^k\) through the \(q^k\) port choices of the fiber.
\end{quote}

To explain this abstraction, we first recall the standard group-valued
voltage-lift point of view (\S\ref{a.2-voltage-lifts-and-the-controller}). In an ordinary voltage lift, one attaches a
group element to each edge of the base graph. Choosing a single path on
the base graph from a start vertex to an end vertex determines a single
product or sum of group elements along that path, and hence a single
displacement in the fiber. Thus, to reach many fiber targets, many base
paths must produce sufficiently different displacements. Even if the
total number of short paths is large, some targets remain unreachable if
their displacements coincide too often.

The route-chart lift point of view is different. Even after fixing a
single route on the controller, one may choose
\(\lambda_i\in\mathbb F_q\) at each step. For a route of length \(5\),
the available control sequences are \[
  (\lambda_1,\ldots,\lambda_5)\in\mathbb F_q^5,
\] of which there are \(q^5\). If the map from these sequences to the
final fiber displacement is a linear isomorphism, then a single
controller route is not merely a single move. It parametrizes all target
points in the fiber.

This idea decomposed the search for the huge final graph into two finite
design problems.

\begin{enumerate}
\def\labelenumi{\arabic{enumi}.}
\tightlist
\item
  \textbf{Controller problem.} Construct a small \(4\)-regular graph
  that joins any pair of controller vertices by a nonbacktracking walk
  of length exactly \(5\).
\item
  \textbf{Fiber chart problem.} Construct a chart over a finite field
  such that, for every reduced word of length \(5\), the corresponding
  \(5\times 5\) controllability matrix is nonsingular.
\end{enumerate}

This decomposition is central to the discovery process. Rather than
directly guessing the whole final graph, the LLM reformulated the task
as the separate construction of a small controller-side graph design and
a finite list of fiber-side linear-algebraic conditions.

\subsection{Success of the fiber chart and the controller
bottleneck}\label{success-of-the-fiber-chart-and-the-controller-bottleneck}

In the proof of concept in Round 024, affine control data were found
over \(\mathbb F_4^4\) such that the controllability matrix associated
with every reduced word of length \(4\) had full rank. This confirmed
that the fiber-chart principle could hold.

Constructing the corresponding controller was harder. For example,
targeting \((16,4)\) requires a \(4\)-regular controller with at least
\(58\) vertices in which all ordered pairs are joined by a
nonbacktracking walk of length exactly \(4\). In Rounds 026--033, many
cyclic voltage lifts, direct searches, and near-controller repairs were
tried, but none reached a certificate.

The important observation at this stage was the identification of a
bottleneck. For the shortest \((16,4)\) target, the controller-side
constraints were too tight.

\subsection{Transition to diameter 5 and the first complete
candidate}\label{transition-to-diameter-5-and-the-first-complete-candidate}

In response to this bottleneck, the search moved to \(s=5\) from Round
034 onward. In Rounds 034--035, the examination of controller candidates
continued with the length extended. Increasing the length to \(5\) gave
more room for controller design, and in Round 036 the first complete
candidate was obtained by combining a universal chart on
\(\mathbb F_4^5\) with a \(132\)-vertex exact-NB-5 controller: \[
  \Delta=16,
  \qquad
  D\le5,
  \qquad
  n=132\cdot4^5=135{,}168 .
\] This was the first candidate graph produced from the route-chart
principle to improve a degree/diameter lower bound.

\subsection{Post-completion generalization search and the update to
the final
controller}\label{post-completion-generalization-search-and-the-update-to-the-final-controller}

In Rounds 037--052, after the first complete candidate had been
obtained, a generalization search was carried out, and the
\(132\)-vertex controller from Round 036 was replaced by the final
\(144\)-vertex controller.

First, in Round 039, combining the same \(132\)-vertex controller with
the \(\mathbb F_3^5\) chart yielded a \((12,5)\) candidate on
\(32{,}076\) vertices. In Rounds 040--048, charts over
\(\mathbb F_5^5\), \(\mathbb F_5^4\), and \(\mathbb F_2^5\), cubic
controllers, and extensions to diameters \(4\) and \(6\) were also
attempted. In many of these attempts, the chart-side condition was
achieved, but no corresponding controller was found.

In Round 049, the user requested a focus on ``ideas for which there is
an a priori mathematical reason to expect that they have the property
currently desired.'' The LLM narrowed the search to a controller
strategy based on \(C_4\times C_6\). In the subsequent Round 050, a
\(144\)-vertex exact-NB-5 controller was obtained from a \(6\)-vertex
\(4\)-regular multibase and \(C_4\times C_6\) voltage coverage.
Composing it with the \(\mathbb F_4^5\) chart yielded \[
  N(16,5)\ge147{,}456 .
\] Then, in Round 052, composing the same \(144\)-vertex controller with
the \(\mathbb F_3^5\) chart yielded \[
  N(12,5)\ge34{,}992 .
\]

It is therefore natural to view the discovery process as having at least
two stages. In the first stage, the route-chart lift construction
principle emerged in Round 023 and reached the first complete candidate
in Round 036. In the second stage, after the post-completion
generalization search encountered a controller bottleneck, the stopping
and reorganization in Rounds 047--050 led to the final \(C_4\times C_6\)
controller.

\section{Transcript-based observations on the discovery
process}\label{transcript-based-observations-on-the-discovery-process}

The analysis in this section is process tracing: it describes changes of
direction, the role of failures, and the emergence of abstraction as
they are observable in the visible transcript. This paper contains no
controlled experiment. Accordingly, this section does not claim how much
any single element contributed to the discovery.

\subsection{Record of unsuccessful approaches and changes of
direction}\label{record-of-unsuccessful-approaches-and-changes-of-direction}

The dialogue log contains many unsuccessful search intervals. These
intervals include the Round 002 attempt to repair reachability by
changing one of the maps assigned to an edge, attempts to modify highly
structured candidate graphs (roughly Rounds 004--015), and the direction
that the LLM tentatively called the ``Hashimoto--de Bruijn packet
construction'' (Rounds 017--018). None of these unsuccessful searches is
directly contained in the final construction. We cannot go so far as to
say that these failures were exploited and produced the final
construction. However, the interpretations the LLM gave of those
failures, and the searches it conducted afterward, can be observed from
the transcript.

A recurring pattern in the early part of the transcript is caution about
improving a nearby candidate. This appears in two forms. The first is
after-the-fact local repair: after a defect has been found, one changes
a small part of the candidate in order to repair that defect. For
instance, in Round 002, the LLM reported that changing one of the maps
assigned to an edge in order to make a particular pair reachable could
break other reachability relations. The second is a more structured
variant: one starts from a highly structured candidate graph and
modifies it in a small or uniform way. In several such attempts, the
reported candidates were close in a numerical sense, but the remaining
obstructions were not localized. The point recorded here is not the
specific structure of those candidates, but the reported failure
pattern. The obstruction did not appear as a small number of defects
that could be removed by a minor adjustment.

In the ``Hashimoto / de Bruijn type packet'' ideas, the obstacle was
formulated as the question of how to carry, within a limited number of
steps, enough information to distinguish each target. In Round 018, the
LLM stated that rather than continuing to repair this direction, one
needed first to design how all targets would be distinguished, as in a
coded multi-window packet or a covering-array quotient.

Thus the transcript contains statements in which the LLM verbalized
missing conditions in several unsuccessful approaches and, in later
searches, appeared to prioritize approaches with a provable covering
principle or a finite certificate over the policy of repairing a nearby
candidate.

\subsection{Constraints contained in the initial
prompt}\label{constraints-contained-in-the-initial-prompt}

The initial prompt in Round 001 contained three kinds of constraints.

\begin{enumerate}
\def\labelenumi{\arabic{enumi}.}
\tightlist
\item
  A clear evaluation goal: obtaining a new record.
\item
  A novelty constraint demanding ideas not essentially contained in the
  relevant literature.
\item
  A constraint suppressing web search during the ideation and
  construction stages.
\end{enumerate}

In the transcript, this combination appears as pressure toward the
search for a new representation rather than a mere re-implementation of
known constructions. As noted in \S\ref{outline-of-the-construction-controller-and-fiber}, however, each element of the
final construction is clearly connected to existing methods. Relative to
existing lift and voltage constructions, the difference emphasized in
this paper does not lie in the individual tools themselves. It lies in
decomposing those tools into controller-side exact-NB-5 coverage and a
fiber-side universal route chart, and then integrating the two parts
into a finite certificate for a diameter proof in the degree/diameter
problem. This point is discussed in detail in \S\ref{discussion}.

\subsection{The user's meta-level interventions and representative
exchanges}\label{the-users-meta-level-interventions-and-representative-exchanges}

In the visible transcript, the author's meta-level interventions
occurred near several junctures at which the level of abstraction of the
search changed. The interventions referred to here did not supply
components of the final construction. They were instructions about how
to proceed with the search, when to stop, novelty requirements, and
requests to explain the prospects of success in advance.

The interventions can be divided broadly into three types. The first is
a stop-and-organize type intervention, such as asking the LLM to pause
and reconsider the direction. The second is a new-principle-focused
intervention, such as asking it to devise a new principle or to
concentrate on mathematically new ideas. The third is an
a-priori-prospect intervention, such as asking it to adopt a direction
for which there is a mathematical reason to expect success.

The following table summarizes exchanges that are especially important
for reading the discovery process. User utterances are quoted briefly,
while the GPT side is summarized.

{\small
\setlength{\LTleft}{0pt}\setlength{\LTright}{\fill}
\begin{longtable}{>{\raggedright\arraybackslash}p{0.06\textwidth}>{\raggedright\arraybackslash}p{0.27\textwidth}>{\raggedright\arraybackslash}p{0.33\textwidth}>{\raggedright\arraybackslash}p{0.20\textwidth}}
\caption{Representative exchanges illustrating the user's meta-level interventions during the search (Section~5.3). User utterances are quoted briefly; the GPT side is summarized.}\label{tab:interventions}\\
\toprule
Round & User prompt (excerpt) & Summary of GPT response & Author's comment \\
\midrule
\endfirsthead
\toprule
Round & User prompt (excerpt) & Summary of GPT response & Author's comment \\
\midrule
\endhead
\bottomrule
\endlastfoot
001 & ``an idea not essentially contained in the literature''; ``do not use web search, in principle'' & After surveying the table of known records, states that it will try a construction that is not an extension of existing methods & Gave general constraints on the search as an initial condition \\
\midrule
003 & ``I do not feel that a good graph can be obtained by local improvement'' & Acknowledges that random lift repair is weak as the main strategy for a large new record, and switches to a construction whose diameter proof closes without repair & Appears to have curbed the drift toward local repair \\
\midrule
017 & also prohibits ``simply trying constructions from methods or concepts a mathematician would think of immediately'' & States that it will ``rebuild the construction principle itself, rather than improving candidates,'' and makes the waste of the distance itself the object of design & A point at which the novelty constraint was explicitly strengthened \\
\midrule
022 & requests concentration on ``ideas with a reason to expect efficient connectivity in advance'' & Discards the ``projective rank-controller direction'' as lacking an a priori basis, and says that a new principle with an explicit covering identity is needed & A point at which the evaluation axis moved from mere near misses to a provable covering principle \\
\midrule
023 & ``please devise a new principle'' & Proposes the controllable route-chart lift, explaining that a single route together with \(q^5\) controls covers the whole fiber & The point at which the abstraction underlying the final construction was first made explicit \\
\midrule
026 & ``For the current 58-vertex controller problem, existing ideas may be used, provided that there is a mathematical reason to expect success.'' & Treats existing voltage-lift ideas as admissible for the isolated controller subproblem, but only under an a-priori-success criterion & Clarified the boundary between the novelty constraint and the controlled use of existing tools \\
\midrule
036 & ``please continue'' & From a \(132\)-vertex exact-NB-5 controller and an \(\mathbb F_4^5\) chart, obtains the first complete candidate \((16,5,135{,}168)\) & A cautionary example showing that large progress can occur even after a simple continuation instruction \\
\midrule
047 & ``please pause and review the direction'' & Summarizes that the generalization search is drifting toward trial-and-error controller search, and proposes a ``sumset-certified controller'' & A point at which the search after the first complete candidate was re-abstracted \\
\midrule
049--050 & A renewed request in Round 049 to avoid ``methods with no mathematical reason to expect success,'' followed by a brief continuation instruction in Round 050 & Focuses the controller search on a \(C_4\times C_6\)-based voltage-lift design, and in Round 050 obtains the final \(144\)-vertex exact-NB-5 controller, yielding the \(147{,}456\)-vertex construction & The final controller update \\
\end{longtable}
}

For Rounds 049--050, the constraint introduced in Round 049 was to
narrow the search to directions with an a priori prospect of success.
This constraint was continued by the brief continuation instruction in
Round 050. For this reason, the paper treats Rounds 049--050 as a single
unit of progress.

The analysis treats Round 023, Round 036, and Rounds
049--050, together with their surrounding exchanges, as focal points.
These are the rounds in which the route-chart principle, the first
complete candidate, and the final controller update respectively
appeared.

Among these, Round 023 and Rounds 049--050 are especially salient in
relation to the user's meta-level interventions. In Round 023, an
explicit intervention requesting a ``new principle'' immediately
preceded the response in which the route-chart lift first appeared. This
route-chart lift became the core of the final construction. In Rounds
049--050, the post-completion generalization search had already
encountered several controller bottlenecks. The request in Round 049 to
concentrate on ``ideas for which there is an a priori reason to expect
that they have the property currently desired'' was followed by a focus
on the \(C_4\times C_6\) type construction used in the final
construction. The brief continuation instruction in Round 050 then led
to the \(C_4\times C_6\) controller.

By contrast, the first complete candidate in Round 036 was obtained
after eight consecutive rounds of brief continuation instructions,
including Round 036 itself.

At the time of Rounds 049--050, however, a separate ChatGPT session was
attempting in parallel to generalize the result around Round 036.
Therefore, although the final idea itself was not obtained in the
separate session, this paper does not exclude the possibility that
verification results, failure information, or the search state in that
separate session indirectly contributed to the progress of Rounds
049--050.

The user's prompts can be broadly classified into three groups. The
first group set the goal of the search, as in Rounds 001, 039, and 052.
The second requested an explanation of results or output to a separate
file, as in Rounds 028, 037, 038, 042, and 051. The third requested
continuation of the search.

There were \(44\) continuation requests. Of these, \(31\) (about
\(70\%\)) added no new mathematical constraint or policy and simply
prompted continuation of the current search with phrases such as
``please proceed'' or ``please continue.'' These occurred in Rounds 002,
005, 006, 008, 010, 012, 013, 014, 015, 016, 018, 019, 020, 021, 024,
029, 030, 031, 032, 033, 034, 035, 036, 040, 041, 043, 044, 045, 046,
048, and 050. The remaining \(13\) continuation requests added
meta-interventions involving the addition, modification, or relaxation
of constraints. The counts are based on the author's classification of
the visible user prompts.

\section{Limitations}\label{limitations}

This paper reports a single case. It does not show that ChatGPT can in
general perform similar mathematical discovery with high probability,
nor that the present method is generally effective. Evaluating the success
rate of this strategy or the effectiveness of meta-level interventions
would require multiple independent runs and ablation studies.

The analysis in this paper is limited to the visible transcript and the
artifacts. We do not observe the LLM's internal reasoning, hidden
chain-of-thought, or the sources actually referenced by the model. Thus,
regarding what the LLM relied on and how it reasoned, we can only infer
from its output proposals, code, and changes of direction.

Finally, the scope of the process analyzed in this paper should be
distinguished from the preparation of Appendix A. Appendix A was not
generated solely by the LLM used in the session analyzed here. It was prepared by the author from files generated in the observed
session, together with later LLM-assisted drafts produced using Claude
Opus 4.7 and Claude Opus 4.8 through Claude Code, and using ChatGPT Pro.
The author then edited this material. The process analyzed in this paper is the process leading to the
construction of the solution, including the generation of the concrete
graph edge lists in the supplementary materials. It does not
include the later process of organizing the mathematical justification
of that construction into the proof exposition given in Appendix A.

\section{Discussion}\label{discussion}

\subsection{Novelty of the graph construction
method}\label{novelty-of-the-graph-construction-method}

The existing framework closest to the present construction is that of
voltage graphs and graph coverings. In a voltage assignment, one assigns
a group element to each directed edge of the base graph, and the product
or sum along a walk describes the movement of the fiber coordinate in
the lift. In the degree/diameter problem as well, voltage-assignment
methods for constructing a large lift from a small base graph have been
widely used \cite{BrankovicEtAl1998a, LozSiran2008}. The controller
certificate in this paper belongs to the same lineage as this standard
voltage-lift idea. Indeed, the \(144\)-vertex controller is the
\(C_4\times C_6\) voltage lift of a \(6\)-vertex multibase, and its
exact-NB-5 property is verified by checking that the voltage sums of
nonbacktracking walks of length \(5\) cover all target voltages.

The final route-chart lift, however, is not the ordinary group-valued
voltage lift of the original controller \(C\). In an ordinary voltage
lift, a single directed edge of the base determines a single
permutation, or a single group displacement, on the fiber. In this
paper, by contrast, a controller edge with label \(t\) is associated
with the \(q\) affine maps \[
  x\mapsto A_t x+\lambda b_t\qquad(\lambda\in\mathbb F_q).
\] Formally, this can be regarded as a permutation voltage lift over the
multigraph obtained by expanding each edge of the controller into \(q\)
control-indexed parallel edges \cite{GrossTucker2001, DalfoEtAl2021}.
From this point of view, the novelty of the construction lies in
choosing the affine permutation voltages so that every reduced route of
length \(5\) parametrizes the entire fiber, and in verifying this fact
by a finite list of rank certificates.

Concretely, writing the label sequence of a fixed controller route as
\(w=t_1\cdots t_5\), the endpoint of the fiber coordinate is given by \[
  x_5=A_wx_0+
  M_w(\lambda_1,\ldots,\lambda_5)^T .
\] Thus the nonsingularity of \(M_w\) means that the single controller
route has a unique control sequence reaching any prescribed target fiber
coordinate. This route-wise controllability condition separates the
diameter check for the huge final graph into the controller-side exact
route coverage and the fiber-side rank condition over all reduced words.
This separation is central to the certificate design in this paper.

\subsection{Decomposition of search execution and search
control}\label{decomposition-of-search-execution-and-search-control}

The search strategy used by the author was to have the LLM carry out
candidate generation and mathematical reasoning, while the author
controlled the hierarchy of the search, stopping decisions, and
acceptance criteria through general interventions.

This can be understood as a working hypothesis that divides the search
into an ``execution process,'' which advances the search within the
problem, and a ``control process,'' which adjusts the granularity and
acceptance criteria of the search from outside the problem. This paper
does not test the effectiveness of this decomposition. It records
instances, including Round 023 and Rounds 049--050, where control-type
interventions and increases in the level of abstraction occur in close
proximity.

\subsection{Relation to recent
cases}\label{relation-to-recent-cases}

Apart from the system-oriented methods listed in \S\ref{introduction}, there have recently
been reports of research-level mathematical results obtained solely
through dialogue with a browser-accessible LLM ordinarily available to
researchers, without a dedicated search engine, formalizer, or
orchestration \cite{Bubeck2025, Salim2025, Verbeken2026, Gowers2026,
Howlett2026}. In particular, \cite{Verbeken2026} is a single-case study
that recorded and analyzed the resolution of an open conjecture by a
consumer LLM across multiple dialogue sessions. It also released the
related artifacts. Methodologically, it is closest to Contribution 2 of
this paper.

Several reports suggest that browser-accessible LLMs can contribute to
new mathematical results. What is especially relevant here is the nature
of human involvement. These reports are not uniform in the role assigned
to the human participant. At least within the released and recorded
materials, some describe processes in which mathematically concrete
directions or pointers were supplied through prompts before the final
result was reached \cite{Salim2025, Verbeken2026}. Others explicitly
state that the human made no essential mathematical contribution, and
that the human role was limited to guidance such as posing the problem,
continuing the search, and requesting confirmation of correctness
\cite{Gowers2026}. The materials released in \cite{Gowers2026}, however,
consist mainly of artifacts and a narrative record, rather than a long
transcript of the kind released with this paper. Moreover, that
case is described as a relatively short exchange on the order of an hour or two, so the time scale of the record differs from the \(52\)-round
search treated in this paper.

With respect to the nature of human involvement, the case reported here
is closer to the latter type than to the former. The author's
contribution was concentrated on controlling the search process rather
than on supplying concrete components of the construction. This control
included demanding novelty, stopping and re-abstracting, and requesting
explanations of the prospects of success in advance (\S\ref{the-discovery-process}, \S\ref{transcript-based-observations-on-the-discovery-process}).

The difference from \cite{Gowers2026} lies instead in the form and time
scale of the record. This paper releases a visible transcript spanning
\(52\) rounds, including unsuccessful search intervals. It also records,
as process tracing, how structured meta-level interventions corresponded
in time to changes in the level of abstraction. These interventions
include stopping and re-abstracting, requesting novelty, and asking for
prospects of success in advance.

Thus, beyond the absence of concrete construction hints
from the author, the distinctive features of this paper are its time
scale and its mode of recording the interaction.

\subsection{Future work}\label{future-work}

This work suggests three perspectives to consider when LLM agents are
used for automatic search on general problems. The first concerns the
role of search control. In the present case, this role was played by the
human author. It remains to be studied how well the search would proceed
if this role were delegated instead to another LLM agent.

The second concerns the separation between search execution and search
control. One possible design is to let a single LLM agent both advance
the search within the problem and adjust the granularity and acceptance
criteria of the search. Another is to separate these two processes
explicitly and assign them to different LLM agents. It remains unclear
what theoretical and practical differences would arise between these
designs.

The third concerns the comparison between lightweight control loops and
dedicated systems. When a powerful foundation model is guided only by a
lightweight control loop, either a human controller or a second LLM, it
is natural to ask how close such a workflow can come to the level
attainable by dedicated systems such as FunSearch
\cite{RomeraParedes2024} or AlphaProof \cite{Hubert2026}. It also remains
to understand in which situations the lightweight-control approach and
the dedicated-system approach are respectively preferable.

This paper provides only a single data point showing that a mathematical
result can be reached through a lightweight workflow without a dedicated
search system. It also remains unclear how far the observations
generalize to other fields requiring exploratory thinking. These
questions are left for future work.

\section{Conclusion}\label{conclusion}

This paper presented a construction proving the new lower bounds \[
  N(12,5)\ge34{,}992,
  \qquad
  N(16,5)\ge147{,}456
\] for the degree/diameter problem. It also gave an account of the
approximately six-day search process through which the construction was
discovered in dialogue with ChatGPT via its standard web interface.

The transcript contains
several points at which the level of abstraction of the search appears
to change after user prompts imposing general constraints. Evaluating
the effectiveness of such control and constraints, and the extent to
which similar methods can handle difficult problems more broadly,
remains a topic for future work.

\section*{Acknowledgments}\label{acknowledgments}

The author thanks Jaeseung Han, Yawara Ishida, and Daisuke Takahashi for
helpful discussions and comments on an earlier draft of this paper, and
for valuable comments on the direction of this work.

\appendix

\section{Mathematical construction and
proofs}\label{appendix-a.-mathematical-construction-and-proofs}

This appendix collects the mathematical constructions needed to prove
\cref{thm:main}.

\subsection{Multigraphs and nonbacktracking
walks}\label{a.1-multigraphs-and-nonbacktracking-walks}

For a finite set \(V\), write \[
  \binom{V}{2}:=\bigl\{\{u,v\}:u,v\in V,\ u\ne v\bigr\} .
\] This is the set of all two-element subsets of \(V\).

\begin{definition}\label{def:multigraph}
A \emph{loopless multigraph} is a triple
\(B=(V,E,\partial)\) where \(V\) and \(E\) are finite sets and
\(\partial:E\to\binom{V}{2}\) is a map. When \(\partial(e)=\{u,v\}\) for
\(e\in E\), we say that \(e\) is an edge joining \(u\) and \(v\).
Distinct edges \(e,e'\in E\) with the same endpoint set \(\{u,v\}\) are
allowed. On the other hand, by the definition of \(\binom{V}{2}\),
\(\partial(e)\) always consists of two distinct vertices, so there are
no loops.
\end{definition}

\begin{definition}\label{def:walk}
A walk of length \(\ell\) in
\(B=(V,E,\partial)\) is a sequence \[
  P=(v_0,e_1,v_1,e_2,\ldots,e_\ell,v_\ell)
\] with \(v_i\in V\), \(e_i\in E\) and \(\partial(e_i)=\{v_{i-1},v_i\}\)
for all \(1\le i\le\ell\). This walk is \emph{nonbacktracking} (NB) if
\[
  e_{i+1}\ne e_i\qquad(1\le i<\ell)
\] is satisfied. For simple graphs, this coincides with the usual
condition \(v_{i+1}\ne v_{i-1}\).
\end{definition}

\subsection{Voltage lifts and the
controller}\label{a.2-voltage-lifts-and-the-controller}

Voltage graphs and regular coverings are standard construction tools. In
this paper we use voltage assignments with values in a finite abelian
group \cite{GrossTucker2001}. In the degree/diameter problem, the
method of giving a voltage assignment to a small base graph to build a
large lift, and reducing the diameter condition to a condition on walks
in the base graph and net voltages, has been used \cite{BrankovicEtAl1998a,
LozSiran2008}.

Let \(B=(V,E,\partial)\) be a loopless multigraph. For each edge
\(e\in E\), with \(\partial(e)=\{u,v\}\), consider the two orientations
\((u,e,v)\) and \((v,e,u)\). We call these the directed edges of \(B\),
and write \(\overrightarrow E(B)\) for the set of all directed edges. We
write the reverse of a directed edge \(\vec e=(u,e,v)\) as
\(\overline{\vec e}=(v,e,u)\).

\begin{definition}\label{def:voltage}
Let \(H\) be a finite abelian group. An
\(H\)-valued voltage assignment is a map \[
  \gamma:\overrightarrow E(B)\to H
\] satisfying, for every directed edge \(\vec e\), \[
  \gamma(\overline{\vec e})=-\gamma(\vec e) .
\] The group element \(\gamma(\vec e)\) assigned to a directed edge
\(\vec e\) is called the voltage of that directed edge.
\end{definition}

The regular lift \(B^\gamma\) of a voltage-equipped multigraph
\((B,\gamma)\) is the following, not necessarily simple, undirected
multigraph. The vertex set is \[
  V(B^\gamma)=V(B)\times H .
\] For each edge \(e\in E(B)\), choose one of its two orientations, say
\(\vec e=(u,e,v)\). For each \(h\in H\), add an undirected edge, with
edge id \((e,h)\), joining \[
  (u,h)
  \quad\text{and}\quad
  (v,h+\gamma(\vec e)).
\] This definition is independent of the chosen orientation. Indeed, if
the reverse orientation is used and \(h'=h+\gamma(\vec e)\), then \[
  h=h'+\gamma(\overline{\vec e}),
\] so the same undirected edge is obtained.

Given a walk \(P=(v_0,e_1,v_1,\ldots,e_\ell,v_\ell)\) and an initial
value \(h_0\in H\), let the directed edge of the \(i\)-th step be
\(\vec e_i=(v_{i-1},e_i,v_i)\) and set \[
  h_i=h_{i-1}+\gamma(\vec e_i) .
\] Then \(h_\ell-h_0=\sum_{i=1}^\ell \gamma(\vec e_i)\) is called the
voltage sum of \(P\).

\begin{lemma}\label{lem:coset}
Let \((B,\gamma)\) be an \(H\)-valued
voltage-equipped loopless multigraph, and set \(C=B^\gamma\). Assume
that for some \(s\ge1\) the following holds. \[
\tag{A.0}
  \forall i,j\in V(B),\ \forall h\in H:\ 
  \text{there is a length-}s\text{ NB walk }i\to j\text{ in }B\text{ with voltage sum }h.
\] Then, for any \((c,d)\in V(C)^2\), there exists a nonbacktracking
walk of length exactly \(s\) from \(c\) to \(d\) in \(C\).
\end{lemma}

\begin{proof}
Write \(c=(i,h_1)\), \(d=(j,h_2)\). Applying (A.0) to
\(h=h_2-h_1\), there is a length-\(s\) NB walk \(P\) in \(B\) from \(i\)
to \(j\) whose voltage sum is \(h_2-h_1\). Lifting \(P\) from the
initial value \(h_1\) by the rule above, the second component of the
endpoint is \(h_1+(h_2-h_1)=h_2\). The NB condition is the condition of
not immediately returning along the same edge id, and is therefore
preserved after lifting. Hence a length-\(s\) NB walk from \(c\) to
\(d\) is obtained in \(C\).
\end{proof}

A simple graph \(C\) with the above property is called an
\textbf{exact-NB-\(s\) controller}.

\subsection{Symmetric alphabets and inverse-consistent
labelings}\label{a.3-symmetric-alphabets-and-inverse-consistent-labelings}

\begin{definition}\label{def:symmalpha}
When a map \(\iota:S\to S\) on a finite set
\(S\) satisfies \(\iota(\iota(t))=t\) for all \(t\in S\), \(\iota\) is
called an involution. If furthermore \(\iota(t)\ne t\) holds for all
\(t\in S\), then \(\iota\) is said to be fixed-point-free. In this
paper, a finite set \(S\) equipped with a fixed-point-free involution is
called a symmetric alphabet, and we write \(\iota(t)\) as \(t^{-1}\).
\end{definition}

A word \(w=t_1\cdots t_s\) of length \(s\) is \emph{reduced} if \[
  t_{i+1}\ne t_i^{-1}\qquad(1\le i<s).
\]

Let \(C\) be a \(\rho\)-regular simple undirected graph and \(S\) a
symmetric alphabet with \(|S|=\rho\). Write \(\overrightarrow E(C)\) for
the set of all directed edges of \(C\). A map
\(\ell:\overrightarrow E(C)\to S\) is an \emph{inverse-consistent
labeling} if it satisfies the following two conditions.

\begin{enumerate}
\def\labelenumi{\arabic{enumi}.}
\tightlist
\item
  For each vertex \(u\in V(C)\), the labels of the \(\rho\) directed
  edges leaving \(u\) take each element of \(S\) exactly once.
\item
  For any directed edge \(u\to v\), \(\ell(v\to u)=\ell(u\to v)^{-1}\).
\end{enumerate}

Then the label sequence of any NB walk in \(C\) is a reduced word. Let
two consecutive directed edges of a walk be \(c_{i-1}\to c_i\) and
\(c_i\to c_{i+1}\), and let the label of the former be \(t\). By
condition 2, the label of the reverse directed edge \(c_i\to c_{i-1}\)
is \(t^{-1}\). By condition 1, there is exactly one directed edge
leaving \(c_i\) with label \(t^{-1}\), so if the next label is
\(t^{-1}\) then \(c_i\to c_{i+1}\) is the directed edge that reverses
the previous edge. This contradicts the NB condition.

By Petersen's 2-factor theorem \cite{Petersen1891}, an even-regular
simple undirected graph admits a 2-factor decomposition. Assign a
distinct inverse pair \(\{t,t^{-1}\}\) of \(S\) to each 2-factor. Choose
an orientation of each cycle, and assign \(t\) to the forward directed
edges and \(t^{-1}\) to the backward directed edges. Then, at each
vertex, the pair \(\{t,t^{-1}\}\) appears once from each 2-factor, and
an inverse-consistent labeling is obtained. In this paper we take two
inverse pairs, namely \(S=\{a,a^{-1},b,b^{-1}\}\).

In the cleaned verification scripts, this step is made completely
deterministic for the \(4\)-regular controller used below. An Euler
circuit is used to two-color the controller edges by parity along the
circuit, giving two \(2\)-factors, and each cycle in the two factors is
then oriented. This produces the inverse-consistent labeling used by the
bundled edge-list generator. The proof of \cref{lem:lift}, however, only needs
the inverse-consistency property and does not depend on this particular
choice of labeling.

\subsection{Route charts and chart
lifts}\label{a.4-route-charts-and-chart-lifts}

Let \(S\) be a symmetric alphabet, \(\mathbb F_q\) a finite field, and
\(s\ge1\).

\begin{definition}\label{def:chart}
A \emph{route chart} on \(\mathbb F_q^s\) is an
assignment, to each \(t\in S\), of \[
  A_t\in\operatorname{GL}_s(\mathbb F_q),
  \qquad
  b_t\in\mathbb F_q^s ,
\] satisfying, for all \(t\in S\), \[
  A_{t^{-1}}=A_t^{-1},
  \qquad
  b_{t^{-1}}=A_t^{-1}b_t .
  \tag{A.1}
\]

For a word \(w=t_1\cdots t_s\) of length \(s\), set \[
  A_w:=A_{t_s}A_{t_{s-1}}\cdots A_{t_1} .
\] For a reduced word \(w=t_1\cdots t_s\) of length \(s\), define the
controllability matrix by \[
  M_w=
  \bigl[
  A_{t_s}\cdots A_{t_2}b_{t_1}\mid
  A_{t_s}\cdots A_{t_3}b_{t_2}\mid
  \cdots\mid
  A_{t_s}b_{t_{s-1}}\mid
  b_{t_s}
  \bigr]
  \in M_s(\mathbb F_q) .
  \tag{A.2}
\] A route chart is \emph{universal} if \(M_w\) is nonsingular for every
reduced word \(w\) of length \(s\).
\end{definition}

\begin{definition}\label{def:lift}
Let \(C\) be a \(\rho\)-regular simple
undirected graph, \(S\) a symmetric alphabet with \(|S|=\rho\),
\(\ell:\overrightarrow E(C)\to S\) an inverse-consistent labeling, and
\((A_t,b_t)_{t\in S}\) a route chart on \(\mathbb F_q^s\). Define the
route-chart lift graph \[
  G=G(C,S,\ell,A,b)
\] as follows. The vertex set is \[
  V(G)=V(C)\times\mathbb F_q^s .
\] For an undirected edge \(\{u,v\}\) of \(C\) with \(\ell(u\to v)=t\),
for each \(x\in\mathbb F_q^s\) and each \(\lambda\in\mathbb F_q\), add
the undirected edge joining the two vertices \[
  (u,x),\qquad (v,A_tx+\lambda b_t) .
\]
\end{definition}

We show that this rule is independent of the orientation chosen for the
undirected edge \(\{u,v\}\). Let the label of the direction \(u\to v\)
be \(t\). For arbitrary \(x\in\mathbb F_q^s\) and
\(\lambda\in\mathbb F_q\), set \[
  y=A_tx+\lambda b_t .
\] Then the rule using the orientation \(u\to v\) adds the edge joining
\((u,x)\) and \((v,y)\).

The reverse direction \(v\to u\) has label \(t^{-1}\). By (A.1), for any
\(\mu\in\mathbb F_q\), \[
  A_{t^{-1}}y+\mu b_{t^{-1}}
  =A_t^{-1}(A_tx+\lambda b_t)+\mu A_t^{-1}b_t
  =x+(\lambda+\mu)A_t^{-1}b_t .
\] Taking \(\mu=-\lambda\) gives \(A_{t^{-1}}y+\mu b_{t^{-1}}=x\). Thus
the reverse orientation adds the same edge \(\{(v,y),(u,x)\}\). Since
\(\lambda\mapsto-\lambda\) is a bijection of \(\mathbb F_q\), the two
orientations give the same family of edges, and the edge-adding rule is
well-defined for the undirected edge. In particular, for any directed
edge \(u\to v\) of \(C\) (label \(t\)), \[
  (u,x)\ \text{and}\ (v,y)\ \text{are adjacent}
  \iff
  \exists\lambda\in\mathbb F_q:\ y=A_tx+\lambda b_t .
\]

\begin{lemma}\label{lem:lift}
With the above notation, suppose the route chart is
universal. Then the following hold.

\begin{enumerate}
\def\labelenumi{\arabic{enumi}.}
\tightlist
\item
  \(G\) is a simple undirected graph.
\item
  \(G\) is \(\rho q\)-regular.
\item
  If \(C\) is an exact-NB-\(s\) controller then \(G\) is connected and
  \(\operatorname{diam}(G)\le s\).
\item
  \(|V(G)|=|V(C)|q^s\).
\end{enumerate}
\end{lemma}

\begin{proof}
The graph is undirected by \cref{def:lift}. We first
prove simplicity and the degree. Since \(C\) is simple, for a directed
edge \(u\to v\) of \(C\) we always have \(u\ne v\). Hence \(G\) has no
loops. By universality, each \(b_t\) is nonzero. Indeed, the word
\(t^s\) is reduced, and the last column of the corresponding \(M_{t^s}\)
is \(b_t\); if \(M_{t^s}\) is nonsingular then \(b_t\ne0\).

Consider a fixed vertex \((u,x)\in V(G)\). There are \(\rho\) directed
edges leaving \(u\), and for each, choosing \(\lambda\in\mathbb F_q\)
yields the candidate neighbor \[
  (v,A_tx+\lambda b_t) .
\] By the argument at the end of \cref{def:lift}, these are all the
neighbors of \((u,x)\in V(G)\). Distinct directed edges go to distinct
controller vertices \(v\) because \(C\) is simple. For the same directed
edge, if \[
  A_tx+\lambda b_t=A_tx+\lambda' b_t ,
\] then \((\lambda-\lambda')b_t=0\), and from \(b_t\ne0\) we get
\(\lambda=\lambda'\). Hence \((u,x)\) has exactly \(\rho q\) distinct
neighbors. Therefore \(G\) is simple and \(\rho q\)-regular.

Next we show the diameter. Take any \((c,x),(d,y)\in V(G)\). Since \(C\)
is exact-NB-\(s\), there is an NB walk of length \(s\) in \(C\), \[
  c=c_0,c_1,\ldots,c_s=d .
\] Let the label of \(c_{i-1}\to c_i\) be \(t_i\). As shown in the
previous subsection, under the inverse-consistent labeling, the label
sequence \(w=t_1\cdots t_s\) of an NB walk is a reduced word.

Choose control parameters \(\lambda_1,\ldots,\lambda_s\in\mathbb F_q\)
along this walk. Starting the fiber coordinate at \(x_0=x\) and setting
\[
  x_i=A_{t_i}x_{i-1}+\lambda_i b_{t_i}
  \qquad(1\le i\le s) ,
\] we obtain inductively \[
  x_s=A_wx+M_w
  \begin{pmatrix}\lambda_1\\ \vdots\\ \lambda_s\end{pmatrix} .
\] By universality \(M_w\) is nonsingular, so the linear equation \[
  M_w
  \begin{pmatrix}\lambda_1\\ \vdots\\ \lambda_s\end{pmatrix}
  =y-A_wx
\] has a solution. Using this solution, we obtain a walk of length \(s\)
from \((c,x)\) to \((d,y)\). In particular \(G\) is connected, and \(\operatorname{diam}(G)\le s\). The
vertex count follows immediately from the definition.
\end{proof}

\begin{remark}
The route-chart lift of \cref{def:lift} can also be
described as a permutation voltage lift over the expanded base
multigraph. That is, replace each directed edge of \(C\) with label
\(t\) by \(q\) directed edges indexed by \(\lambda\in\mathbb F_q\), and
assign to that directed edge the permutation of \(\mathbb F_q^s\) \[
  \pi_{t,\lambda}(x)=A_t x+\lambda b_t .
\] By condition (A.1), replacing the control parameter \(\lambda\) by
\(-\lambda\) makes the affine map on the reverse directed edge the
inverse permutation. Hence the route-chart lift is an instance of a
graph lift by a permutation voltage assignment \cite{GrossTucker2001,
DalfoEtAl2021}. In this paper, in addition to this permutation lift,
we impose the nonsingularity of the controllability matrix for all
reduced routes of length \(s\), thereby reducing the diameter proof to a
finite list of rank certificates.
\end{remark}

\subsection{The 144-vertex
controller}\label{a.5-the-144-vertex-controller}

Let the vertex set of \(B\) be \(\{0,1,2,3,4,5\}\) and the voltage group
be \[
  H=C_4\times C_6=\mathbb Z/4\mathbb Z\times\mathbb Z/6\mathbb Z .
\] Each edge is as in \cref{tab:base}, with the orientation \(u\to v\) in the
table taken as the forward direction.

{\centering
\begin{longtable}{rrrrr}
\caption{Base multigraph \(B\) and voltage assignment.}\label{tab:base}\\
\toprule
edge id & \(u\) & \(v\) & \(C_4\) component & \(C_6\) component \\
\midrule
\endfirsthead
\toprule
edge id & \(u\) & \(v\) & \(C_4\) component & \(C_6\) component \\
\midrule
\endhead
\bottomrule
\endlastfoot
0 & 0 & 2 & 1 & 0 \\
\midrule
1 & 0 & 3 & 2 & 4 \\
\midrule
2 & 0 & 4 & 2 & 0 \\
\midrule
3 & 0 & 4 & 2 & 5 \\
\midrule
4 & 1 & 2 & 1 & 2 \\
\midrule
5 & 1 & 3 & 3 & 5 \\
\midrule
6 & 1 & 5 & 1 & 1 \\
\midrule
7 & 1 & 5 & 1 & 0 \\
\midrule
8 & 2 & 4 & 2 & 2 \\
\midrule
9 & 2 & 5 & 1 & 5 \\
\midrule
10 & 3 & 4 & 1 & 5 \\
\midrule
11 & 3 & 5 & 3 & 1 \\
\end{longtable}
}

Each vertex occurs four times in \cref{tab:base}, and there are no loops.
Hence \(B\) is a loopless \(4\)-regular multigraph. Setting
\(C:=B^\gamma\), \[
  V(C)=\{0,1,2,3,4,5\}\times C_4\times C_6 ,
\] so \(|V(C)|=6\cdot4\cdot6=144\).

\begin{proposition}\label{prop:controller}
\(C\) is a simple \(4\)-regular connected
graph and an exact-NB-5 controller.
\end{proposition}

\begin{proof}
Regularity follows from the definition of the regular
lift and the \(4\)-regularity of \(B\). As for simplicity, since \(B\)
has no loop, the lift has no loop. The only parallel pairs of base edges
are edges \(2\) and \(3\), and edges \(6\) and \(7\); in each pair, the
voltages are different. For two base edges joining the same two base
vertices to give the same lift edge, their voltages must be equal, so no
multiple edge arises.

The base multigraph \(B\) is connected by inspection of \cref{tab:base}.
Connectivity of the lift follows from the fact that the closed-walk
voltages of the base generate all of \(H=C_4\times C_6\). For example,
the voltage of \[
  0\xrightarrow{e_3}4\xrightarrow{e_2^{-1}}0
\] is \((2,5)-(2,0)=(0,5)\), which generates the \(C_6\) component.
Also, the voltage of \[
  0\xrightarrow{e_0}2\xrightarrow{e_8}4\xrightarrow{e_2^{-1}}0
\] is \((1,0)+(2,2)-(2,0)=(1,2)\). Together with \((0,5)\), this
generates all of \(H\). Hence \(C\) is connected.

For the exact-NB-5 property, by \cref{lem:coset} it suffices to verify that,
for any \(i,j\in\{0,\ldots,5\}\) and any \(h\in C_4\times C_6\), there
is a length-\(5\) NB walk \(i\to j\) in \(B\) with voltage sum \(h\).
The number of length-\(5\) NB walks from each starting vertex of \(B\)
is at most \(4\cdot3^4=324\), and over all starting vertices it is at
most \(6\cdot324=1944\), so a complete enumeration is possible.
The supplementary \texttt{scripts/*/verify\_candidate.py} enumerates these
by DFS and verifies that, for each of the \(36\) ordered base pairs, all
\(24\) voltages of \(H\) are realized, i.e.\ all \(36\cdot24=864\)
pair--voltage combinations of condition (A.0).
\end{proof}

\subsection{\texorpdfstring{Universal route charts for
\(s=5\)}{A.6 Universal route charts for s=5}}\label{a.6-universal-route-charts-for-s5}

In this section we set \(S=\{0,1,2,3\}\) with the correspondence
\(0=a\), \(1=a^{-1}\), \(2=b\), \(3=b^{-1}\). The involution is
\(0\leftrightarrow1\), \(2\leftrightarrow3\). There are
\(4\cdot3^4=324\) reduced words of length \(5\).

\subsubsection{\texorpdfstring{Chart on
\(\mathbb F_3^5\)}{A.6.1 Chart on \textbackslash mathbb F\_3\^{}5}}\label{a.6.1-chart-on-mathbb-f_35}

Let \(\mathbb F_3=\{0,1,2\}\). We use the following matrices and
vectors.

\[
A_0=\begin{pmatrix}
1&2&0&2&0\\
1&2&1&2&0\\
1&2&1&1&0\\
2&0&0&0&0\\
1&2&2&0&1
\end{pmatrix},\quad
A_1=\begin{pmatrix}
0&0&0&2&0\\
2&2&1&2&0\\
2&1&0&0&0\\
0&1&2&0&0\\
1&0&1&0&1
\end{pmatrix},
\] \[
A_2=\begin{pmatrix}
2&1&0&0&0\\
0&1&1&0&0\\
2&2&0&1&0\\
2&0&1&0&0\\
2&2&1&1&1
\end{pmatrix},\quad
A_3=\begin{pmatrix}
1&2&0&1&0\\
2&2&0&1&0\\
1&2&0&2&0\\
0&1&1&2&0\\
2&1&2&1&1
\end{pmatrix},
\] \[
b_0=\begin{pmatrix}2\\2\\2\\0\\0\end{pmatrix},\quad
b_1=\begin{pmatrix}0\\1\\0\\0\\1\end{pmatrix},\quad
b_2=\begin{pmatrix}2\\1\\2\\0\\2\end{pmatrix},\quad
b_3=\begin{pmatrix}1\\0\\1\\0\\2\end{pmatrix}.
\]

\begin{proposition}\label{prop:a10}
The correspondence \((A_t,b_t)_{t\in S}\) is
a universal route chart on \(\mathbb F_3^5\).
\end{proposition}

\begin{proof}
The supplementary
\texttt{scripts/12\_5\_34992/verify\_candidate.py} verifies the
following with arithmetic over \(\mathbb F_3\).

\begin{enumerate}
\def\labelenumi{\arabic{enumi}.}
\tightlist
\item
  \(A_0A_1=A_1A_0=I_5\), \(A_2A_3=A_3A_2=I_5\).
\item
  \(A_1b_0=b_1\), \(A_0b_1=b_0\), \(A_3b_2=b_3\), \(A_2b_3=b_2\).
\item
  For all \(324\) reduced words \(w\) of length \(5\), the matrix
  \(M_w\) of (A.2) has rank \(5\).
\end{enumerate}

These are exactly the conditions of \cref{def:chart} and universality.
\end{proof}

\subsubsection{\texorpdfstring{Chart on
\(\mathbb F_4^5\)}{A.6.2 Chart on \textbackslash mathbb F\_4\^{}5}}\label{a.6.2-chart-on-mathbb-f_45}

Let \(\mathbb F_4=\mathbb F_2[\alpha]/(\alpha^2+\alpha+1)\), with the
encoding \(0=0\), \(1=1\), \(2=\alpha\), \(3=\alpha+1\). We use the
following matrices and vectors.

\[
A_0=\begin{pmatrix}
2&0&1&2&0\\
1&2&2&0&0\\
3&0&0&3&0\\
0&0&1&3&0\\
2&1&2&2&2
\end{pmatrix},\quad
A_1=\begin{pmatrix}
2&0&3&2&0\\
0&3&1&1&0\\
1&0&3&0&0\\
2&0&1&2&0\\
1&2&2&3&3
\end{pmatrix},
\] \[
A_2=\begin{pmatrix}
3&1&3&2&0\\
0&1&0&0&0\\
1&0&2&0&0\\
2&1&0&0&0\\
3&2&2&0&2
\end{pmatrix},\quad
A_3=\begin{pmatrix}
0&3&0&3&0\\
0&1&0&0&0\\
0&2&3&2&0\\
3&1&1&2&0\\
0&2&3&3&3
\end{pmatrix},
\] \[
b_0=\begin{pmatrix}2\\1\\2\\0\\2\end{pmatrix},\quad
b_1=\begin{pmatrix}2\\1\\3\\1\\2\end{pmatrix},\quad
b_2=\begin{pmatrix}2\\1\\0\\2\\2\end{pmatrix},\quad
b_3=\begin{pmatrix}2\\1\\1\\3\\2\end{pmatrix}.
\]

\begin{proposition}\label{prop:a11}
The correspondence \((A_t,b_t)_{t\in S}\) is
a universal route chart on \(\mathbb F_4^5\).
\end{proposition}

\begin{proof}
The supplementary
\texttt{scripts/16\_5\_147456/verify\_candidate.py} verifies, with
\(\mathbb F_4\) arithmetic using the above encoding, the same three
kinds of conditions as in \cref{prop:a10}. That is, the inverse
relations (A.1) and the rank-\(5\) condition for all \(324\)
controllability matrices hold.
\end{proof}

\subsection{Proof of the main
theorem}\label{a.7-proof-of-the-main-theorem}

\begin{theorem}\label{thm:bounds}
The following inequalities hold: \[
  N(12,5)\ge34{,}992,
  \qquad
  N(16,5)\ge147{,}456 .
\]
\end{theorem}

\begin{proof}
Fix the \(144\)-vertex controller \(C\) of \S\ref{a.5-the-144-vertex-controller} and an
inverse-consistent labeling of its directed edges, which exists by
\S\ref{a.3-symmetric-alphabets-and-inverse-consistent-labelings}.

First let \(q=3\) and use the universal route chart of \S\ref{a.6.1-chart-on-mathbb-f_35}. By
\cref{prop:controller}, \(C\) is exact-NB-5, and by \cref{prop:a10} the chart
is universal. By \cref{lem:lift}, \[
  G_3=G(C,S,\ell,A,b)
\] is simple connected \(4q=12\)-regular with \(\operatorname{diam}(G_3)\le5\),
and satisfies \[
  |V(G_3)|=|V(C)|q^5=144\cdot3^5=34{,}992 .
\]

Similarly, letting \(q=4\) and using the universal route chart of
\S\ref{a.6.2-chart-on-mathbb-f_45}, by \cref{lem:lift}, \[
  G_4=G(C,S,\ell,A,b)
\] is simple connected \(16\)-regular with \(\operatorname{diam}(G_4)\le5\), and
satisfies \[
  |V(G_4)|=144\cdot4^5=147{,}456 .
\] Hence the two claimed lower bounds follow.

Furthermore, \(34{,}992\) exceeds the \((12,4)\) Moore upper bound
\(17{,}569\), and \(147{,}456\) exceeds the \((16,4)\) Moore upper bound
\(57{,}857\). Hence the diameters of these graphs cannot be at most
\(4\), and in both cases the diameter is exactly \(5\).
\end{proof}

\subsection{Verification and
reproducibility}\label{a.8-verification-and-reproducibility}

Every computer-dependent part of the proof in this paper consists of
verifying finitely many conditions. The supplementary materials provide the
data and scripts in two packages, which record the same graphs.

\begin{itemize}
\tightlist
\item
  \texttt{route\_chart\_verified\_graphs\_package/} is a frozen
  provenance artifact, not a cleaned software release. It preserves the
  process by which the construction was obtained. Some generating
  scripts contain unstandardized elements, including default output
  paths from the execution environment used during the dialogue and
  minor differences in the format of the edge-list header. The
  natural-language description files are kept as provenance only and are
  not parsed by the certificate-level or generation-level verification.
  This package is not used in Checks 1--3 below. The optional Check 0
  only compares the clean package against it in order to record
  provenance equivalence.
\item
  \texttt{route\_chart\_clean\_package/} is cleaned and runnable, and is
  the package to execute. For each parameter set, the construction data
  are collected in a single machine-readable certificate,
  \texttt{certificates/*/certificate.json}. This certificate records the
  voltage assignment on the base multigraph, given in \cref{tab:base},
  together with the universal route chart on \(\mathbb F_q^5\) defined
  in \S\ref{a.6-universal-route-charts-for-s5}, namely the transition matrices \(A_t\) and vectors \(b_t\).
  The corresponding edge list is bundled as
  \texttt{graphs/*/final\_graph\_edges.tsv.gz}, and the step-by-step
  procedure is given in \texttt{route\_chart\_clean\_package/VERIFY.md}.
\end{itemize}

The objects used in the mathematical proof are the construction data made
explicit in Appendix A and the corresponding finite list of certificate
conditions.
The runnable verification is organized as the following four checks. The
commands are run from the root of
\texttt{route\_chart\_clean\_package/}.

\begin{enumerate}
\def\labelenumi{\arabic{enumi}.}
\setcounter{enumi}{-1}
\item
  \textbf{Provenance equivalence (optional).}
  \texttt{verify\_provenance\_equivalence.py} confirms that the certificate
  and the edge list in the clean package describe the same construction as
  the provenance package. The comparison of certificates is semantic, since
  the clean package uses a unified JSON format while the provenance package
  keeps the original text files. This ties the runnable artifact to the recorded
  discovery.
\item
  \textbf{Certificate-level verification.}
  \texttt{scripts/*/verify\_candidate.py} loads the certificate and verifies
  the voltage coverage of length-\(5\) nonbacktracking walks on the base
  multigraph, the simplicity, \(4\)-regularity, and connectivity of the
  regular lift controller, the inverse-consistent labeling, the direct
  exact length-\(5\) nonbacktracking coverage of all ordered vertex
  pairs on the \(144\)-vertex controller, and the
  rank-\(5\) condition of the controllability matrix for all \(324\) reduced
  words of length \(5\). The expected output has the following form.

{\small
\begin{verbatim}
Base voltage coverage over C4 x C6: OK
Controller: OK; 144 vertices, simple 4-regular,
  exact labeled NB length-5 coverage
GF(3)^5 universal chart: OK; 324 reduced words checked,
  min rank 5; inverse-symbol data verified
Certificate complete: degree 12, diameter <=5,
  vertices = 144*3^5 = 34992
\end{verbatim}
}

  This is the output for \((12,5)\). The \((16,5)\) case prints
  \texttt{GF(4)\^{}5}, \texttt{degree\ 16}, and
  \texttt{144*4\^{}5\ =\ 147456} in the corresponding positions.

\item
  \textbf{Generation-level identity check.}
  \texttt{scripts/*/generate\_final\_graph.py} rebuilds the edge list from
  the certificate, writing it to a path specified on the command line. The
  top-level \texttt{verify\_graphs\_identical.py} then checks that this graph
  and the bundled
  \texttt{route\_chart\_clean\_package/graphs/*/final\_graph\_edges.tsv.gz}
  are identical, in two parts: it compares the SHA-256 hash of the edge rows,
  excluding the header, and it checks agreement of the edge sets after
  normalizing the orientation of each edge. The first part verifies agreement
  of the edge rows in the same order. The second verifies graph-theoretic
  identity independently of file format or row order.
\item
  \textbf{Graph-level independent check.} \texttt{verify\_with\_igraph.py}
  reads the bundled edge list directly and checks the absence of self-loops,
  the absence of multiple edges, the number of vertices, the number of edges,
  the regular degree, the handshake equality, the connectivity, and the exact
  diameter computed by igraph. This script uses only the edge list,
  and depends on neither the certificate nor the generating scripts. The exact
  diameter computation is an additional check. The diameter-\(5\) proof of this paper does not depend on
  it.
\end{enumerate}

In summary, the mathematical proof relies on the explicit data in
Appendix A and on the finite certificate conditions checked in Check 1.
Check 2 verifies that the bundled edge lists are exactly the graphs
generated from those data. Check 3 is an additional, construction-free
verification of the bundled edge lists.

\section{Structure of the supplementary
materials}\label{appendix-b.-structure-of-the-supplementary-materials}

The supplementary materials contain the
visible transcript of the discovery process, the script that produces
that transcript from the conversation log, the construction package
preserved as provenance, a cleaned and runnable version of that package,
and the scripts for provenance comparison, identity
checking, and independent graph-level verification.

\begin{itemize}
\tightlist
\item
  \texttt{dump\_ddp\_visible.py}: a script that extracts, from the
  conversation log obtained via the official ChatGPT download function,
  the ordinary user prompt bubbles and assistant answer bubbles visible
  in the browser UI. Other non-chat records in the export, such as hidden
  reasoning, tool calls, tool outputs, and system messages, are not
  included in the released transcript. The only export metadata retained
  are per-message timestamps and model identifiers, together with a brief
  conversation-level header, kept as annotations.
\item
  \texttt{extracted/DDP\_visible\_transcript.md}: the visible transcript
  extracted by \texttt{dump\_ddp\_visible.py}. It targets Rounds 001--052
  relevant to the present discovery and verification.
\item
  \texttt{route\_chart\_verified\_graphs\_package/}: the construction
  package produced by the LLM during the dialogue, preserved as a
  provenance artifact rather than as the runnable verification package.
  Details are given in \S\ref{a.8-verification-and-reproducibility} of Appendix A.
\item
  \texttt{route\_chart\_clean\_package/}: a cleaned, runnable version of
  \texttt{route\_chart\_verified\_graphs\_package/}, intended to be
  executed by verifiers. Details are given in \S\ref{a.8-verification-and-reproducibility} of Appendix A.
\item
  \texttt{verify\_provenance\_equivalence.py}: a script that confirms that
  the clean package describes the same construction as the provenance
  package, by a semantic comparison of the certificates and a content
  comparison of the bundled edge lists.
\item
  \texttt{verify\_graphs\_identical.py}: a script that verifies that the
  graph generated by \texttt{generate\_final\_graph.py} and the graph bundled
  under \texttt{route\_chart\_clean\_package/graphs/} are identical. It checks
  the SHA-256 of the edge rows excluding the header, and the agreement of the
  normalized edge set.
\item
  \texttt{verify\_with\_igraph.py}:
  independent verification script. It checks the bundled graph for the
  absence of self-loops, the absence of multiple edges, the number of
  vertices, the number of edges, the regular degree, the handshake
  equality, and the connectivity, and computes the exact diameter with igraph. The diameter computation is an optional additional check,
  and the diameter-\(5\) proof of this paper does not depend on this script.
\end{itemize}

\bibliographystyle{abbrvnat}
\bibliography{references}

\end{document}